\documentclass[11pt]{article}
\usepackage[dvips]{graphicx}
\usepackage[T2A]{fontenc}
\usepackage[cp1251]{inputenc}
\usepackage[russian]{babel}
\usepackage{amssymb}
\usepackage{amsmath}
\usepackage{amscd}
\usepackage{amsthm}

\begin{document}

\renewcommand{\refname}{References}
\renewcommand{\abstractname}{Abstract}
\renewcommand{\proofname}{Proof:}
\newtheorem{te}{Theorem}[subsection]
\newtheorem{lem}{Lemma}[subsection]
\newtheorem{cor}{Corollary}
\theoremstyle{definition}
\newtheorem{df}{Definition}[subsection]

\begin{center}

{\Large\bf Geometric Theory of Parshin's Residues. Toric Neighborhoods of Parshin's Points.}

\bigskip

{\footnotesize by Mikhail Mazin, University of Toronto.}

\end{center}

\begin{abstract}

The paper consist of two parts. In the first part we introduce flags of lattices and associated
injective systems of (non-normal) cones and projective systems of (non-normal) affine toric
varieties. We study the associated field of multidimensional Laurent power series. In the second
part we use the resolution of singularities techniques to study the geometry near a complete flag
of subvarieties and the Parshin's residues. The first part plays the role of a standard coordinate
neighborhood for Parshin's points.

\end{abstract}

\section{Flags of Lattices and the Associated Systems of Toric Varieties.}

\subsection{Ordered Abelian Groups.}

\begin{df}
An Abelian group $G$ subdivided into subsets $G=G_-\sqcup \{0\}\sqcup G_+$ is called an {\it
ordered abelian group} if

(1) $a\in G_- \Rightarrow -a\in G_+;$

(2) $G_+$ is a semigroup.
\end{df}

We write $a>b$ if $a-b\in G_+$ and $a<b$ if $a-b\in G_-.$ Easy to see that this gives a total
ordering of $G.$ The elements of $G_+$ are called positive and the elements of $G_-$ are called
negative.

\begin{df}
A subgroup $H\subset G$ is called {\it isolated} if for any two positive elements $g_1>g_2>0$ such
that $g_1\in H$ it follows that $g_2\in H.$
\end{df}

It follows immediately that if $h>0,$ $g>0,$ $h\in H,$ and $g\not\in H$ than $g>h.$

\begin{te}
Let $H^1\subset G$ and $H^2\subset G$ be two isolated subgroups. Then either $H_1\subset H_2$ or
$H_2\subset H_1.$
\end{te}

\begin{df}
Let $G=H^k\supset H^{k-1}\supset\dots\supset H^0=\{0\}$ be the tower of all isolated subgroups of
$G.$ Then $k$ is the {\it rank of $G.$}
\end{df}

\begin{te}\label{Ordered Groups}
Suppose that $G$ is isomorphic to $\mathbb Z^n$ and has rank $n.$ Let $G=H^n\supset
H^{n-1}\supset\dots\supset H^0=\{0\}$ be the isolated subgroups of $G.$ Then for all $k$ $H^k$ is
isomorphic to $\mathbb Z^k$ and the order on $G$ is isomorphic to the lexicographic order with
respect to any basis $(e_n,\dots,e_1)$ of $G$ such that for all $k$ $(e_k,\dots, e_1)$ is a basis
in $H^k.$ (The lexicographic order range the elements first with respect to the coefficient at
$e_n,$ than $e_{n-1},$ etc.)
\end{te}

\begin{df}
Let $G$ be an ordered abelian group of rank greater than $1.$ Let $G'\subsetneq G$ be its maximal
proper isolated subgroup. Than $H_+:=G_+\backslash G'$ is called the {\it upper half-space} in $G.$
\end{df}

For groups of rank $1$ we set $H_+=G_+$ and for $G=\{0\}$ we set $H_+=\{0\}.$

\subsection{Flags of Lattices. Injective Systems of Cones and Projective Systems of Toric
Varieties.}

Let $L^n$ be an ordered abelian group of rank $n$ isomorphic to $\mathbb Z^n.$ Let $L^n\supset
\hat{L}^{n-1}\supset\dots\supset \hat{L}^1\supset L^0$ be the isolated subgroups of $L^n.$ Let
$L^i\subset \hat{L}^i$ for $1\le i\le n-1$ be subgroups of full rank such that $L^n\supset
L^{n-1}\supset\dots\supset L^0$.

\begin{df}
$L^n\supset L^{n-1}\supset\dots\supset L^0$ is called a {\it flag of lattices.}
\end{df}

We use the multiplicative notations for the operation in $L^n.$

Let $H^k_+$ be the upper half-space in $L^k$ for $k=0,\dots,n.$ We are interested in the semigroup
$L=H^0_+\cup\dots\cup H^n_+$ (i.e. we take $0,$ then all the positive elements of $L^1,$ then all
the positive elements in $L^2$ which are not in $\hat{L}^1,$ etc.). Note, that $L$ is not finitely
generated. In particular, it doesn't correspond to any algebraic variety. However, one can consider
$L$ as a union of countably many cones, which are finitely generated.

\begin{df}
Let $\hat{\cal C}$ be the set of all simple cones $\hat{C}$ in $L^n$ such that for all $k=1,\dots
n$ exactly $k$ generators of $\hat{C}$ belong to $\hat{L}^k_+$ (i.e. one of the generators is
exactly the generator of $\hat{L}^1_+,$ another one belong to $\hat{L}^2_+,$ etc.).
\end{df}

\begin{df}
Let $\hat{C}\in \hat{\cal C}$ and let $(x_1,\dots,x_n)$ be the generators of $\hat{C}$ such that
$x_i\in\hat{L}_k$ for $i=1,\dots,n.$ Then we call $(x_1,\dots,x_n)$ the {\it standard generators}
of $\hat{\cal C}.$
\end{df}

The cones from $\hat{\cal C}$ are not subsets of $L.$ So we need to intersect them with $L.$

\begin{df}
Let ${\cal C}=\{C=\hat{C}\cap L:\hat{C}\in\hat{\cal C}\}.$
\end{df}

\begin{lem}
Elements of ${\cal C}$ are finitely generated.
\end{lem}

\begin{proof}
Let $C\in {\cal C}$ and let $(x_1,\dots,x_n)$ be the standard generators of $\hat{C}.$ Let
$k_1,\dots,k_n$ be the smallest positive integers such that $x_i^{k_i}\in C.$ Denote
$y_i:=x_i^{k_i}.$ Let $\hat{P}\subset \hat{C}$ be the subset in $\hat{C}$ consisting of all the
monomials $x_1^{d_1}\dots x_n^{d_n}$ such that $0\le d_i\le k_i$ for $i=1,\dots,n.$ Let
$P=\hat{P}\cap C.$ $P$ is obviously finite. Let us prove that it generates $C.$

Any element $x_1^{m_1}\dots x_n^{m_n}\in\hat{C}$ can be factored in the following way:
$x_1^{m_1}\dots x_n^{m_n}=y_1^{l_1}\dots y_n^{l_n}x_1^{d_1}\dots x_n^{d_n},$ where $x_1^{d_1}\dots
x_n^{d_n}\in \hat{P},$ $l_i\ge 0,$ and $d_i\neq 0$ if and only if $m_i\neq 0$ for $i=1,\dots,n.$
Therefore, we only need to prove that if $x_1^{m_1}\dots x_n^{m_n}\in C$ then $x_1^{d_1}\dots
x_n^{d_n}\in C$ as well.

Let $x_1^{m_1}\dots x_n^{m_n}\in C.$ Let $j$ be the biggest number such that $m_j\neq 0,$ i.e.
$x_1^{m_1}\dots x_n^{m_n}=x_1^{m_1}\dots x_n^{m_j}$ and $m_j\neq 0.$ Then $d_n=\dots=d_{j+1}=0$ and
$d_j\neq 0$ as well. Since $x_1^{m_1}\dots x_n^{m_j}\in\ C,$ it follows that $x_1^{m_1}\dots
x_n^{m_j}\in L^j$ and $x_1^{m_1}\dots x_n^{m_j}\notin \hat{L}^{j-1}.$ Since $y_1,\dots,y_j\in L^j,$
$x_1^{d_1}\dots x_n^{d_n}\in L^j$ as well. Also, since $d_j\neq 0,$ $x_1^{d_1}\dots x_n^{d_n}\notin
\hat{L}_{j-1}.$ So, $x_1^{d_1}\dots x_n^{d_n}\in L$ and, therefore, $x_1^{d_1}\dots x_n^{d_n}\in
C.$
\end{proof}

\begin{lem}\label{Cones containing finite sets of positive elements}
Let $K\subset L^n_+$ be a finite set of positive elements in $L^n.$ Then there exist
$\hat{C}\in\hat{\cal C}$ such that $K\subset\hat{C}.$
\end{lem}

\begin{proof}
It is enough to prove that if $\hat{C}\subset\hat{\cal C}$ and $a\in L^n_+$ then there exist
another cone $\hat{C}'\in\hat{\cal C}$ such that $\hat{C}\subset\hat{C}'$ and $a\in\hat{C}'.$
Indeed, using this fact one can start from any cone in $\hat{\cal C}$ and add elements of $K$ one
by one.

Let $(x_1,\dots,x_n)$ be the standard generators of $\hat{C}.$ Let $a=x_1^{k_1}\dots x_n^{k_n}.$
Let $j$ be the biggest number such that $k_j\neq 0.$ Note that $k_j>0,$ otherwise $a$ is not
positive. Take the cone $\hat{C}'$ generated by $(x_1,\dots,x_{j-1},x'_j,x_{j+1},\dots,x_n),$ where
$x'_j=x_1^{min(k_1,0)}\dots x_{j-1}^{min(k_{j-1},0)}x_j$ (i.e. $x'_j$ is obtained from
$a=x_1^{k_1}\dots x_j^{k_j}$ by removing all the factors with positive powers and replacing
$x_j^{k_j}$ by $x_j$).

$\hat{C}'\in \hat{\cal C}$ because $x'_j\in L_j$ and the transition matrix from $(x_1,\dots,x_n)$
to $(x_1,\dots x_{j-1},x'_j,x_{j+1},\dots,x_n)$ is integrally invertible. Indeed,
$x_j=x_1^{-min(k_1,0)}\dots x_{j-1}^{-min(k_{j-1},0)}x'_j.$

We need to check that $\hat{C}\subset\hat{C}'$ and $a\in\hat{C}'.$ Indeed, all the generators of
$\hat{C}$ except for $x_j$ are also generators of $\hat{C}$ and $x_j\in\hat{C}'$ because
$x_j=x_1^{-min(k_1,0)}\dots x_{j-1}^{-min(k_{j-1},0)}x'_j$ and $-min(k_i,0)\ge 0$ for
$i=1,\dots,j-1.$ Also, $a=x_1^{k_1}\dots x_n^{k_n}=x_1^{k_1}\dots x_j^{k_j}=x_1^{k_1}\dots
x_{j-1}^{k_{j-1}}(x_1^{-min(k_1,0)}\dots x_{j-1}^{-min(k_{j-1},0)}x'_j)^{k_j}=x_1^{k_1-k_j
min(k_1,0)}\dots x_{j-1}^{k_{j-1}-k_j min(k_{j-1},0)}(x')_j^{k_j}.$ Since $k_j>0,$ it follows, that
$k_i-k_j min(k_i,0)\ge 0.$ Therefore, $a\in\hat{C}'.$
\end{proof}

\begin{cor}
Let $G\subset L^n_+$ be any finitely generated semigroup in $L^n_+.$ Then there exist
$\hat{C}\in\hat{\cal C}$ such that $G\subset\hat{C}.$
\end{cor}

\begin{cor}
Let $G\subset L$ be any finitely generated semigroup in $L.$ Then there exist $C\in{\cal C}$ such
that $G\subset C.$
\end{cor}

\begin{cor}
$L=\bigcup\limits_{C\in {\cal C}}C.$
\end{cor}

\begin{cor}
Both $\hat{\cal C}$ and ${\cal C}$ are injective systems under inclusion.
\end{cor}

Although the elements of ${\cal C}$ are not cones in the usual sense, we still call them cones. If
$C\in {\cal C}$ then the corresponding element $\hat{C}\in \hat{\cal C}$ is called the
normalization of $C.$

There is the correspondence between the elements of ${\cal C}$ and non-normal affine toric
varieties $T_C.$ (One can construct the variety $T_C$ as follows: semigroup algebra of $C$ is
finitely generated and therefore correspond to the affine variety which we denote by $T_C.$) We
denote the set of toric varieties $T_C$ for all $C\in {\cal C}$ by ${\cal T}.$

Any $C\in {\cal C}$ has exactly one $k$-dimensional face which span $L_k,$ for any $k=0,\dots,n.$
Therefore, one can identify one of the $k$-dimensional orbits in all $T\in {\cal T}$ for each
$k=0,\dots,n.$ We denote these orbits $T^0,T^1,\dots,T^n.$ So, $T^i\subset T$ for all $T\in{\cal
T}$ and $i=0,\dots,n.$ Note also, that $T^k\cup T^{k-1}$ is an affine subvariety in every $T\in
{\cal T}$ for $i=1,\dots,n.$ However, any union of more than two consequent orbits is not a
subvariety.

Let $C,C'\in {\cal C}$ and $C'\subset C.$ Then there is the natural map $\phi_{C,C'}:T_C\to
T_{C'}.$ Note, that $\phi_{C,C'}$ is identity on $T^0,T^1,\dots,T^n.$ Since the maps $\phi_{C,C'}$
goes in the back direction, ${\cal T}$ is a projective system with respect to these maps.

Take $T_C\in {\cal T}.$ Although it is not normal, its normalization $\widetilde{T}_C$ is
isomorphic to $\mathbb C^n.$ The standard generators $(x_1,\dots,x_n)$ of the cone $\hat{C}$ give
coordinates on $\widetilde{T}.$ We'll use $(x_1,\dots,x_n)$ as coordinates for $T$ as well
(although some points are going to be glued together). We call these coordinates the {\it standard
coordinates} for the varieties $T\in {\cal T}.$

Switching from $T_C$ to $T_{C'}$ corresponds to the change of coordinates

$$
\begin{array}{l}
x'_1=x_1 \\
x'_2=x_1^{d_{12}}x_2 \\
\vdots \\
x'_n=x_1^{d_{1n}}\dots x_{n-1}^{d_{(n-1)n}}x_n,
\end{array}
$$

\noindent where

$$
\left(
\begin{array}{lllll}
1      & d_{12} & d_{13} & \dots & d_{1n} \\
0      & 1      & d_{23} & \dots & d_{2n} \\
\vdots & \vdots & \vdots &       & \vdots \\
0      & 0      & 0      & \dots & 1
\end{array}
\right)
$$

\noindent is the transition matrix from the standard generators of the cone $\hat{C}$ to the
standard generators of the cone $\hat{C}'.$

The following lemma is important in the study of the Laurent series of $L:$

\begin{lem}
Let $\hat{C}'\in \hat{\cal C}$ and $S\subset\hat{C}',$ $S\neq\emptyset.$ Than $S$ contains its
minimal element $s_{min}$ and there exist another cone $\hat{C}\in\hat{\cal C}$ such that $S\subset
\hat{C}+s_{min}$.
\end{lem}

\begin{proof}\label{Minimal element}
Let $(x_1,\dots,x_n)$ be the standard generators of $\hat{C}'.$ Note, that according to the Theorem
\ref{Ordered Groups} the order in $L_n$ coincides with the lexicographic order with respect to the
basis $(x_1,\dots,x_n)$ (first, with respect to the power of $x_n,$ then the power of $x_{n-1},$
etc.). Let $m_n$ be the smallest integer such that $x_1^{k_1}\dots x_{n-1}^{k_{n-1}}x_n^{m_n}\in S$
for some integers $k_1,\dots,k_{n-1}.$ Let $m_{n-1}$ be the smallest integer such that
$x_1^{k_1}\dots x_{n-2}^{k_{n-2}}x_{n-1}^{m_{n-1}}x_n^{m_n}\in S$ for some integers
$k_1,\dots,k_{n-2}.$ Continuing in the same way we get the integers $m_1,\dots,m_n$ such that
$s_{min}=x_1^{m_1}\dots x_n^{m_n}\in S$ is the minimal element in $S.$

Let $S'=\{a-s_{min}: a\in\hat{C}'$ and $a>s_{min}\}.$ Clearly, $S\subset S'+s_{min}.$ Therefore, it
is enough to find $\hat{C}\in\hat{\cal C}$ such that $S'\subset \hat{C}.$

Let $y_1=x_1,$ $y_2=x_1^{-m_1}x_2,$ $\dots,$ $y_n=x_1^{-m_1}\dots x_{n-1}^{-m_{n-1}}x_n.$ Easy to
see that $y_k\in \hat{L}^k_+$ for $k=1,\dots,n.$ Moreover, the transition matrix from
$(x_1,\dots,x_n)$ to $(y_1,\dots,y_n)$ is upper triangular with units on diagonal. Therefore, the
cone generated by $(y_1,\dots, y_n)$ is simple and belong to $\hat{\cal C}$. Denote this cone by
$\hat{C}.$ Let's prove that $S'\subset\hat{C}.$ Indeed, let $s=x_1^{l_1}\dots x_n^{l_n}\in S'.$ It
follows from the definition of $S'$, that $l_k\ge -m_k$ for $k=1,\dots,n$ and that the last
non-zero power, say $l_p,$ is positive (otherwise $s<0$). Then $s=y_p(x_1^{l_1+m_1}\dots
x_{p-1}^{l_{p-1}+m_{p-1}}x_p^{l_p-1}).$ Note, that all the powers in the last formula are
non-negative and that $x_k\in \hat{C}$ for all $k=1,\dots n.$ Therefore, $s\in\hat{C}.$
\end{proof}

\subsection{Laurent Power Series.}

\begin{df}
Let $F(L)^*$ be the set of all formal infinite linear combinations $f=\sum\limits_{p\in L^n}f_p p$
of the elements of $L^n$ such that
\begin{enumerate}
\item There exist at least one point in the torus $T^n$ such that $f$ is convergent at this point.
\item There exist a cone $\hat{C}\in {\cal \hat{C}}$ and an element $p_0\in L^n$ such that the Newton polyhedron of $f$ is a subset of $\hat{C}+p_0.$
\end{enumerate}
Let $F(L)=F(L)^*\sqcup \{0\}$ (i.e. we add the series with zero coefficients to $F(L)^*$). We call
$F(L)$ the set of Laurent series of $L.$
\end{df}

\begin{lem}\label{valuation}
Let $f\in F(L)^*.$ Then one can choose the cone $\hat{C}$ and an element $p_0\in L^n$ in such a way
that the Newton polyhedron of $f$ is a subset of $\hat{C}+p_0$ and $f_{p_0}\neq 0.$
\end{lem}

\begin{proof}
Let $Sup_f=\{p\in L^n: f_p\neq 0\}.$ Since $f\in F(L),$ $Sup_f\subset \hat{C}'\in\hat{\cal C}.$ Now
we just apply the Lemma \ref{Minimal element} to $Sup_f$ to get $p_0$ and $\hat{C}.$
\end{proof}

Note, that the $p_0\in L^n$ from Lemma \ref{valuation} is unique for each $f\in F(L)^*.$ Basically,
it is the smallest element in $L^n$ such that the corresponding coefficient of $f$ is not zero.
Therefore, one gets a map $\nu:F(L)^*\to L^n.$

\begin{lem}\label{Existance of a common cone for several functions}
Let $f_1,\dots,f_m\in F(L).$ Then there exist $T\in {\cal T}$ such that $\frac{f_i}{\nu(f_i)}$
converge to holomorphic non-zero functions in a neighborhood of the origin in the normalization
$\widetilde{T}$ of $T$ for $i=1,\dots,m$ (here $(x_1,\dots,x_n)$ are the standard coordinates on
$T$).
\end{lem}

\begin{proof}
Let $\hat{C}_i\in {\cal \hat{C}}$ be such that the Newton polyhedron of $f_i$ is a subset of
$\hat{C}_i+\nu(f_i).$ Let $\hat{C}\in \hat{\cal C}$ be such that $\hat{C}_i\subset \hat{C}$ for all
$i=1,\dots,m$ (the existence of $\hat{C}$ follows from the Lemma \ref{Cones containing finite sets
of positive elements}). Let $C=\hat{C}\cap L,$ $T=T_C,$ and $(x_1,\dots,x_n)$ be the standard
coordinates in $T.$ Then, for any $i=1,\dots,m,$ $\frac{f_i}{\nu(f_i)}$ is a Taylor series in
$(x_1,\dots,x_n)$ converging at least at one point in $T^n\subset T$ and, therefore, converging in
a neighborhood of zero in $\widetilde{T}.$ Moreover, since the coefficient of $f_i$ at $\nu(f_i)$
is not zero, $f_i$ is not zero at the origin.
\end{proof}

{\bf Remark.} Note, that $\nu(f_i)$'s are monomials in $(x_1,\dots,x_n).$ So, $f_1,\dots,f_m$ are
{\it almost monomial} in the neighborhood of the origin in $\widetilde{T},$ i.e. each of them is
equal to a monomial multiplied by a non-zero holomorphic function.

\begin{te}\label{F(L) is a Field}
$F(L)$ is a field and $\nu:F(L)^*\to L^n$ is a homomorphism of the multiplicative group $F(L)^*$ of
$F(L)$ to the ordered abelian group $L^n.$
\end{te}

\begin{proof}
The theorem follows immediately from the previous lemma. Indeed, if $f,g\in F^*(L)$ then there
exist $T\in {\cal T}$ with coordinates $(x_1,\dots,x_n)$ such that both $\frac{f}{\nu(f)}$ and
$\frac{g}{\nu(g)}$ converge to holomorphic non-zero functions in a neighborhood of the origin in
$\widetilde{T}.$ Let $\nu(f)=x_1^{k_1}\dots x_n^{k_n}$ and $\nu(g)=x_1^{m_1}\dots x_n^{m_n}.$ Then
$\frac{fg}{\nu(f)\nu(g)},$ $\frac{f+g}{x_1^{\min(k_1,m_1)}\dots x_n^{\min(k_n,m_n)}},$ and
$\frac{f^{-1}}{\nu(f)^{-1}}$ also converge to holomorphic functions in a neighborhood of the
origin. Moreover, $\frac{fg}{\nu(f)\nu(g)}$ and $\frac{f^{-1}}{\nu(f)^{-1}}$ don't vanish at the
origin. Therefore, $fg,f^{-1},f+g\in F(L),$ $\nu(fg)=\nu(f)\nu(g),$ and $\nu(f^{-1})=\nu(f)^{-1}.$
\end{proof}

\begin{df}
A homomorphism of the multiplicative group of a field $F$ to an ordered abelian group $G$ is called
a valuation on $F.$
\end{df}

According to the Theorem \ref{F(L) is a Field}, the field $F(L)$ is endowed with a valuation.

\begin{lem}\label{expansions of rational functions}\
\begin{enumerate}
\item Let $r$ be a meromorphic function on $T^n$ (which is the same as a meromorphic function on any $T\in {\cal T}$). Then there exist a Laurent series $f\in F(L)$ normally converging to $r$ in $\{0<|x_1|<\epsilon_1,0<|x_2|<\epsilon_2,\dots,0<|x_n|<\epsilon_n\}\subset T^n,$ where $(x_1,\dots,x_n)$ are the standard coordinates on one of the varieties $T\in {\cal T}.$
\item Let $T\in{\cal T}.$ Let $r$ be a meromorphic function on a neighborhood of the origin in $T.$ Then there exist a Laurent series $f\in F(L)$ normally converging to $r$ in $\{0<|x_1|<\epsilon_1,0<|x_2|<\epsilon_2,\dots,0<|x_n|<\epsilon_n\}\subset T^n,$ where $(x_1,\dots,x_n)$ are the standard coordinates on one of the varieties $T'\in {\cal T}.$
\end{enumerate}
\end{lem}

\begin{proof}
Both statements immediately follows from the Theorem \ref{F(L) is a Field}. Indeed, the Taylor
expansions of functions holomorphic at the origin in the normalization of any $T\in{\cal T}$ belong
to $F(L)$ and $F(L)$ is closed under division by non-zero elements.
\end{proof}

$F(L)$ is basically the field of functions, which are meromorphic in a "good"\ neighborhood of the
flag $T^0,T^1,\dots,T^n.$ Note, that $F(L)$ and the valuation on it depend only on $L^n,$ and don't
depend on other elements of the flag $L^0\subset L^1\subset\dots\subset L^n.$ However, if one wants
to consider the subalgebra in $F(L)$ of functions which are holomorphic in a "good"\ neighborhood
of the flag $T^0,T^1,\dots,T^n,$ it will depend on the whole flag $L^0\subset
L^1\subset\dots\subset L^n$ (or, equivalently, on the semigroup $L$).

\begin{df}
Let $O(L)\subset F(L)$ be the subset of all series $f=\sum\limits_{p\in L^n} f_p p$ in $F(L)$ such
that $f_p\neq 0 \Rightarrow p\in L.$
\end{df}

{\bf Remark.} Let $f\in F(L)$ and $f=\nu(f)\phi(x_1,\dots,x_n)$ where $(x_1,\dots,x_n)$ are the
coordinates on the appropriate $T\in {\cal T}$ and $\phi$ is a holomorphic function in the
neighborhood of the origin in the normalization $\widetilde{T}\simeq\mathbb C^n.$ Then the series
of $f$ belong to $O(L)$ if and only if $\nu(f)\ge 0$ and $\phi$ can be pushed forward to $T.$

\begin{lem}
$O(L)\subset F(L)$ is a subalgebra.
\end{lem}

\begin{proof}
Follows immediately from the definition.
\end{proof}

{\bf Remark.} Let $C\in {\cal C}$ be a cone. Denote by $O(C)$ the ring of series of elements of
$C,$ converging in a neighborhood of the origin in $T_C.$ In other words, $O(C)$ is the ring of
germs of analytic functions on $T_C$ at the origin. Then $O(L)=\bigcup\limits_{C\in {\cal C}}O(C).$
Therefore, $O(L),$ in a sense, is the ring of germs of functions on the inverse limit of ${\cal T}$
at the origin.

\begin{lem}\label{almost normal}
Any $T\in {\cal T}$ satisfy the following continuation property. Let $U\subset T$ be any open
subset and $f:U\to\mathbb C$ be any continuous function, holomorphic in the complement to an
analytic subset $\Sigma\subset U$ of codimension $1.$ Then $f$ is holomorphic in $U.$
\end{lem}

\begin{proof}
It is enough to check that $f$ is holomorphic in a neighborhood of every point in $U.$ Moreover,
every point in $T$ has a neighborhood isomorphic to a neighborhood of the origin in a variety
constructed in the same way as $T,$ but for a different flag of lattices. Therefore, it is enough
consider the case when $U$ contains the origin and to check that $f$ is holomorphic in a
neighborhood of the origin.

Let $(x_1,\dots, x_n)$ be the standard coordinates on $T$ and $C\in {\cal C}$ be the corresponding
cone. Let $\pi:\widetilde{T}\to T$ be the normalization map. Then $\tilde{f}=f\circ \pi$ is
continuous function in a neighborhood of the origin in $\widetilde{T}\simeq\mathbb C^n$ and
holomorphic in the compliment to the analytic subset $\pi^{-1}(\Sigma)$ of codimension $1.$
Therefore, by Riemann Extension Theorem, it is holomorphic in this neighborhood and can be expanded
into the Taylor series $\tilde{f}=\sum a_{\overline{k}}x^{\overline{k}}$ in it.

Since $f$ is continuous in a neighborhood of the origin in $T,$ it follows that if
$a_{\overline{k}}\neq 0$ then $x^{\overline{k}}\in C.$ Therefore, $\sum
a_{\overline{k}}x^{\overline{k}}\in O(C)$ and $f$ is regular at the origin in $T.$
\end{proof}

\subsection{Changes of Variables.}

Let $(f_1,\dots,f_n)$ be an $n$-tuple of functions from $F(L)$ and
$x_1=\nu(f_1),\dots,x_n=\nu(f_n).$ Suppose that:

\begin{enumerate}
\item $(x_1,\dots,x_n)$ are the standard coordinates on a variety in ${\cal T};$
\item\label{generators of C_f belong to O(L)} for any integers $k_1,\dots,k_n$ such that $x_1^{k_1}\dots x_n^{k_n}\in L$ we have $f_1^{k_1}\dots f_n^{k_n}\in O(L).$
\end{enumerate}

Let $L_f^n$ be the lattice of monomials in $f_1,\dots,f_n.$ Easy to see that the restriction
$\phi:=\nu|_{L_f^n}:L_f^n\to L^n$ is an isomorphism. Consider the flag of lattices $L_f^0\subset
L_f^1\subset\dots\subset L_f^n,$ where $L_f^k=\phi^{-1}(L^k)$ for $k=0,\dots,n$ and the order on
$L_f^n$ is induced by $\phi.$ Denote by $L_f$ the corresponding semigroup, ${\cal C}_f$ the
corresponding system of cones and ${\cal T}_f$ the corresponding system of toric varieties.

\begin{te}\label{change of variables}
There exist a toric variety $T_f\in {\cal T}_f$ with the standard coordinates $(g_1,\dots,g_n),$
and a toric variety $T\in {\cal T}$ such that the $n$-tuple $(g_1,\dots,g_n)$ provides an analytic
isomorphism of a neighborhood of the origin in $T$ to a neighborhood of the origin in $T_f$ (note,
that $g_1,\dots,g_n$ are monomials in $f_1,\dots,f_n$).
\end{te}

\begin{proof}
According to the Theorem \ref{Existance of a common cone for several functions} there exist a toric
variety $T\in {\cal T}$ such that $\frac{f_k}{x_k}$ converge to a non-zero holomorphic function in
a neighborhood of zero of the normalization $\widetilde{T}$ for $1\le k\le n.$ Let
$(y_1,\dots,y_n)$ be the standard coordinates in $T$ and $C\in {\cal C}$ be the corresponding cone.
The transition matrix from $(x_1,\dots,x_n)$ to $(y_1,\dots,y_n)$ is integer and upper-triangular
with units on diagonal. For instance, let $y_k=x_1^{d_{1k}}\dots x_{k-1}^{d_{(k-1)k}}x_k$ for
$k=1,\dots,n.$ Let $g_k=f_1^{d_{1k}}\dots f_{k-1}^{d_{(k-1)k}}f_k$ for $k=1,\dots,n.$ Let
$\hat{C}_f\in \hat{\cal C}_f$ be the cone generated by $(g_1,\dots,g_n)$ and let $T_f\in {\cal
T}_f$ be the corresponding toric variety. Then $\nu(g_k)=y_k$ for $k=1,\dots,n$ and
$\frac{g_k}{y_k}$ converge to holomorphic functions in a neighborhood of the origin in the
normalization $\widetilde{T}$ of $T$ for $k=1,\dots,n.$ Let $h_k=\frac{g_k}{y_k}.$ Then
$\frac{\partial g_i}{\partial y_j}(0)=\frac{\partial h_iy_i}{\partial y_j}(0)=h_i(0)\frac{\partial
y_i}{\partial y_j}(0)+\frac{\partial h_i}{\partial y_j}(0)y_i(0)=h_i(0)\delta_{i,j}.$ Since
$h_i(0)\neq 0$ for $i=1,\dots,n,$ the Inverse Function Theorem is applicable. Therefore,
$(g_1,\dots,g_n)$ provides an isomorphism $\hat{G}:\widetilde{U}\to\widetilde{U_f}$ of a polydisk
$\widetilde{U}$ with center at the origin of $\widetilde{T}$ to a neighborhood of the origin
$\widetilde{U_f}$ in $\widetilde{T}_f.$

Let $\pi(\widetilde{U})=U\subset T$ and $\pi_f(\widetilde{U_f})=U_f\subset T_f.$ We need to show,
that the isomorphism $\hat{G}:\widetilde{U}\to\widetilde{U_f}$ can be pushed down to an isomorphism
$G:U\to U_f.$ According to the Lemma \ref{almost normal}, it is enough to show that $G$ is a
homeomorphism. Moreover, since the topology of $U$ ($U_f$ respectively) is the factor topology of
$\pi:\widetilde{U}\to U$ ($\pi_f:\widetilde{U_f}\to U_f$ respectively), it is enough to show, that
$G$ is a bijection.

First, we construct the map $G$ and then prove that it is surjective and injective. We get that it
is regular for free, but the above arguments allow us to avoid considering the inverse map. The
condition \ref{generators of C_f belong to O(L)} provides that the generators of $C_f$ are regular
(and, therefore, well defined) on $U.$ Indeed, they are regular on $\widetilde{U}$ and they belong
to $O(L)$ by condition \ref{generators of C_f belong to O(L)}. Surjectivity follows immediately
from the diagram. All we need to prove is the injectivity.

Easy to see, that $\hat{G}$ maps the coordinate cross to the coordinate cross and, in particular,
maps $\widetilde{T}^k\cap \widetilde{U}$ to $\widetilde{T}^k_f\cap \widetilde{U}_f,$  where
$\widetilde{T}^k$ and $\widetilde{T}^k_f$ are respectively the coordinate subspaces
$\widetilde{T}^k=\{y_n=\dots=y_{k+1}=0\}\subset \widetilde{T}$ and
$\widetilde{T}^k_f=\{g_n=\dots=g_{k+1}=0\}\subset \widetilde{T}_f.$ Therefore, for any point $x\in
U$ the number of preimages of $x$ in $\widetilde{U}$ is bigger or equal to the number of preimages
of $G(x)$ in $\widetilde{U}_f.$ But the preimages of $x$ are mapped injectively by $\hat{G}$ to the
preimages of $G(x).$ Therefore, $\hat{G}|_{\pi^{-1}(x)}$ is a bijection to $\pi_f^{-1}(G(x)),$
which implies the injectivity of $G.$
\end{proof}

Switching from ${\cal T}$ to ${\cal T}'$ is called a {\it change of variables}. Easy to see, that a
change of variables gives an isomorphism from $F(L)$ to $F(L_f).$

The same arguments works for any bigger cone $C'\supset C.$ More precisely, we have the following

\setcounter{cor}{0}

\begin{cor}
Let $C\in {\cal C}$ and $C_f\in {\cal C}_f$ be as in the Theorem \ref{change of variables}. Let
$C'\in {\cal C}$ be such that $C\subset C'.$ Let $C'_f\in {\cal C}_f$ be such that the transition
matrix from the standard generators of the $\hat{C}'_f$ to the standard generators of the
$\hat{C}_f$ is the same as the transition matrix for the generators of $\hat{C}'$ and $\hat{C}$
respectively. Then there exist neighborhoods of the origins in $T_{C'}$ and $T_{C'_f}$ $U'$ and
$U'_f$ and the unique isomorphism $G':U'\to U'_f$ such that $G\circ\phi_{C',C}=\phi_{C'_f,C_f}\circ
G'$ on $U'.$
\end{cor}

So, the change of variables provides isomorphisms between the neighborhoods of the origins in the
elements of ${\cal T}$ and ${\cal T}_f$ respectively, at least starting from some $T\in {\cal T}$
and $T_f\in {\cal T}_f.$ These isomorphisms commute with the maps $\phi_{C',C}:T_{C'}\to T_{C'_f}.$

\subsection{Residue.}

One can consider the free one-dimensional module $\Omega(L)$ over $F(L)$ with the generator
$\omega_{T^n}=\frac{dx_1}{x_1}\wedge\dots\wedge\frac{dx_n}{x_n}.$ Note, that $\omega_{T^n}$ doesn't
depend on the choice of coordinates $(x_1,\dots,x_n).$ Then for every element of
$\omega\in\Omega(L)$ there exist a toric variety $T\in{\cal T}$ with coordinates $(x_1,\dots,x_n)$
such that $\omega=x_1^{d_1}\dots x_n^{d_n}\phi dx_1\wedge\dots\wedge dx_n$ where $\phi$ is a
holomorphic non-zero function in a neighborhood of the origin in $\widetilde{T}.$ We call the
elements of $\Omega(L)$ the germs of meromorphic $n$-forms at the flag of orbits
$T^n,T^{n-1}\dots,T^0.$ According to the Lemma \ref{expansions of rational functions}, any
meromorphic form on a neighborhood of the origin in any $T\in{\cal T}$ can be expanded into a power
series from $\Omega(L).$

\begin{df}
Let $\omega\in\Omega(L),$ $\omega=(\sum\limits_{p\in L^n}a_pp)\omega_{T^n}.$ Then the residue of
$\omega$ is given by $res(\omega)=a_1.$
\end{df}

\begin{lem}\label{residue as integral}
$$
res(\omega)=\frac{1}{(2\pi i)^n}\int\limits_{\tau^n}\omega,
$$
\noindent where $\tau^n=\{(x_1,\dots,x_n)\in T^n: |x_1|=\epsilon_1,\dots,|x_n|=\epsilon_n\}$ for
$(x_1,\dots,x_n)$ --- coordinates on a toric variety $T\in {\cal T},$ and $\epsilon_i$ are small
enough, so that $\omega$ converges on $\tau^n.$ The orientation on $\tau^n$ is provided by the form
$\frac{1}{(2\pi i)^n}\omega_{T^n}|_{\tau^n}$ (note, that this form is real).
\end{lem}

\begin{proof}
Follows immediately from the Fubini's Theorem and the formula for the one-dimensional residues.
\end{proof}

Suppose that the $n$-tuple of functions $(f_1,\dots,f_n)$ defines a change of variables. Let
$\omega\in \Omega(L)$ be a germ of a meromorphic form. Let $\hat{G}:\widetilde{U}\to
\widetilde{U}_f$ be an isomorphism of neighborhoods of the origins in the normalizations of the
appropriate toric varieties $T\in{\cal T}$ and $T_f\in{\cal T}_f$ provided by the change of
variables (see Theorem \ref{change of variables}). According to the Corollary $1$ from the Theorem
\ref{change of variables}, one can assume that $\omega$ converges to a meromorphic form in
$U\subset T.$ Then one can push forward $\omega$ using the isomorphism $\hat{G}.$ The result can be
expanded as a power series from $\Omega(L_f).$ Therefore, we get a map from $F_*:\Omega(L)\to
\Omega(L_f).$ Easy to see, that this map is an isomorphism of $F(L)$-modules ($\Omega(L_f)$ is an
$F(L)$-module via the isomorphism of $F(L)$ and $F(L_f)$ provided by the change of variables).

\begin{lem}
Changes of variables doesn't change the residue, i.e. $res(F_*(\omega))=res(\omega).$
\end{lem}

\begin{proof}
Follows from the Lemma \ref{residue as integral} and the observation that $\hat{G}$ restricts to an
isomorphism of $\widetilde{U}\cap T^n$ and $\widetilde{U}_f\cap T^n_f$ and this isomorphism is
homotopic to the identity map (here we identify $T^n$ and $T^n_f$ with the standard tori via the
standard coordinates on $T$ and $T_f$).
\end{proof}

\subsection{Algebraic functions.}

Let $X$ be an analytic (algebraic) variety and let $g_0,\dots,g_k$ be regular functions on $X.$ Consider the equation $g_0+g_1t\dots+g_kt^k=0.$ Let $U=X\backslash(sing(X)\cup \{g_k=0\}\cup \Sigma),$ where $\Sigma=\{Dis=0\},$ where $Dis$ is the discriminant of the polynomial $g_0+g_1t\dots+g_kt^k.$ Then there exist the $k$-sheeted covering $p:W\to U$ and a regular function $f$ on $W,$ such that for every point $x\in U$ the values of $f$ on the preimage $p^{-1}(x)$ are exactly the roots of the equation $g_0+g_1t\dots+g_kt^k.$ In such a situation, we say that $f$ is an algebraic function on $X.$ We say that $\Sigma$ is the divisor of branching of $f.$ %We need some standard results about the algebraic functions and branched coverings.

\begin{lem}\label{solution is meromorphic}
Let $X$ be an analytic (algebraic) variety and let $g_0,\dots,g_k$ be regular functions on $X.$
Suppose that there is an open subset $U\subset X,$ such that $X\backslash U$ is a finite union of
subvarieties of codimension $1,$ and a holomorphic function $f$ on $U$ such that $f$ satisfy the
equation $g_0+g_1t\dots+g_kt^k=0$ on $U.$ Then $f$ can be continued to a meromorphic function on
$X.$
\end{lem}

\begin{proof}
Consider the function $\tilde{f}=g_kf.$ Easy to see that it is holomorphic on $U$ and satisfy the
integral equation $g_0g_k^{k-1}+g_1g_k^{k-2}t+\dots+g_{k-1}t^{k-1}+t^k=0$ on $U.$ Let's prove that
$\tilde{f}$ is locally bounded on $X,$ i.e. for any point $x\in X$ there exist a neighborhood $V$
of $x,$ such that $\tilde{f}$ is bounded in $V\cap U.$ Indeed, assume that it is not true. The
coefficients $g_ig_k^{k-i-1}$ are regular on $X,$ and, therefore, locally bounded. So, there exist
a neighborhood $V$ of $x$ and a constant $M>1,$ such that $|g_ig_k^{k-i-1}|<M$ for $0\le i\le k-1.$
Since $\tilde{f}$ is not bounded in $V\cap U,$ there exist a point $y\in V\cap U,$ such that
$|\tilde{f}(y)|>kM.$ Then $|\tilde{f}^k(y)|>kM|\tilde{f}^{k-1}(y)|.$ But
$|g_i(y)g_k^{k-i-1}(y)\tilde{f}^i(y)|<M|\tilde{f}^i(y)|\le M|\tilde{f}^{k-1}(y)|$ for $0\le i\le
k-1.$ So,
$|g_0g_k^{k-1}+g_1g_k^{k-2}\tilde{f}(y)+\dots+g_{k-1}\tilde{f}^{k-1}(y)|<|\tilde{f}^k(y)|.$
Therefore, $\tilde{f}$ doesn't satisfy the equation
$g_0g_k^{k-1}+g_1g_k^{k-2}t+\dots+g_{k-1}t^{k-1}+t^k=0$ at $y,$ which is a contradiction.

Let $p:\widetilde{X}\to X$ be the normalization. Then $\tilde{f}\circ p$ is holomorphic in the
complement of a finite union of subvarieties of codimension $1$ in the normal variety
$\widetilde{X}$ and is locally bounded in $\widetilde{X}.$ Therefore, it is regular on
$\widetilde{X}.$ The normalization map $p$ is a birational isomorphism, so $\tilde{f}$ is
meromorphic on $X.$ Finally, $f=\frac{\tilde{f}}{g_k},$ therefore $f$ is also meromorphic on $X.$
\end{proof}

Let now $g_0,\dots,g_k\in F(L).$ One can choose $T\in {\cal T}$ such that $g_k$ and the
discriminant of the equation $g_kt^k+\dots+g_0=0$ are almost monomial in a polydisk $D$ with center
at the origin of the normalization $\widetilde{T}\simeq\mathbb C^n.$ Let $X=D\cap T^n.$ Let $p:W\to
X$ be the corresponding $k$-sheeted covering. Suppose that $W$ is connected (which is equivalent to
say that the equation $g_kt^k+\dots+g_0=0$ is irreducible). We need the following Lemma:

\begin{lem}\label{coverings}
Let $N=k!.$ Consider the map $P:\mathbb C^n\to \mathbb C^n$ given by
$P(x_1,\dots,x_n)=(x_1^N,\dots,x_n^N).$ Let $W'=P^{-1}(X).$ Then there exist a map $\phi:W'\to W$
such that $p\circ\phi=P.$
\end{lem}

\begin{proof}
According to the classical theory, the connected coverings of $X$ are classified up to isomorphism
by the subgroups of the fundamental group $\pi_1(X)$ and the number of sheets corresponds to the
index of the subgroup. In more details, the induced homomorphism $p_*:\pi_1(W)\to \pi_1(X)$ is
injective and its isomorphism type defines the isomorphism type of the covering $p:W\to X.$

Since $\pi_1(X)\simeq\mathbb Z^n,$ it follows, that $p_*(\pi_1(W))\subset \pi_1(X)$ is a sublattice
of full rank and index $k.$ Let $(a_1,\dots,a_n)$ be the basis of $\pi_1(X)$ corresponding to the
loops going around coordinate hyperplanes in positive direction. Then $a_i^N\in p_*(\pi_1(W))$ for
all $i.$ Consider now the covering $P:W'\to X.$ Easy to see, that $P_*(\pi_1(W'))\subset \pi_1(X)$
is generated by $(a_1^N,\dots,a_n^N).$ Therefore, $\pi_1(W')\subset \pi_1(W)$ (here we identified
$\pi_1(W)$ and $\pi_1(W')$ with their images in $\pi_1(X)$). Therefore, there exist a covering
$\phi:W''\to W$ such that $\phi_*(\pi_1(W''))=\pi_1(W')\subset \pi_1(W)\subset \pi_1(X).$ But then
the composition $p\circ\phi:W''\to X$ and $P:W'\to X$ correspond to the same subgroup in $\pi_1(X)$
and, therefore, isomorphic.
\end{proof}

According to the Lemmas \ref{coverings} and \ref{solution is meromorphic}, the equation
$g_kt^k+\dots+g_0=0$ has a meromorphic solution $f$ in $P^{-1}(D).$ Moreover, $f$ is holomorphic in
$W'=P^{-1}(D)\cap (\mathbb C^*)^n.$ Therefore, we have the following theorem:

\begin{te}
Let $g_0,\dots,g_k\in F(L)$ be such that the equation $g_kt^k+\dots+g_0=0$ is irreducible. Let
$N=k!.$ Then there exist a toric variety $T\in {\cal T}$ with the coordinates $(x_1,\dots,x_n),$
and the integers $m_1,\dots,m_n$ such that the root $f$ of the equation $g_kt^k+\dots+g_0=0$ can be
written in the form of the Piezo series
$f=\sqrt[N]{x_1}^{m_1}\dots\sqrt[N]{x_n}^{m_n}\sum\limits_{i_1,\dots,i_n\ge
0}f_{i_1,\dots,i_n}\sqrt[N]{x_1}^{i_1}\dots\sqrt[N]{x_n}^{i_n}$ converging in a neighborhood of the
origin in the normalization $\widetilde{T}.$
\end{te}

\section{Toric Neighborhoods of Parshin's points.}

\subsection{Branched coverings and Generic components of the preimage of a hypersurface.}

\begin{df}
Let $X$ and $Y$ be non-empty pure-dimensional algebraic (analytic) varieties of the same dimension.
An algebraic (analytic) map $f:Y\to X$ is called a {\it branched covering} if it is proper,
surjective, and is a local isomorphism at a generic point (i.e. there is an open dense subset
$U\subset Y$ consisting of smooth points of $Y$ such that for every point $y\in U$ $f(y)$ is a
smooth point of $X$ and the differential of $f$ has full rang at $y$).
\end{df}

Note, that a composition of two branched coverings is again a branched covering.

\begin{df}
Let $f:Y\to X$ be a branched covering. Let $H\in X$ be a hypersurface in $X$. Then the {\it generic
component of the preimage} (shortly, the {\it generic preimage}) $H_f\subset Y$ is the union of
those irreducible components $\hat{H}_i$ of the full preimage $f^{-1}(H)$ for which the restriction
$f|_{\hat{H}_i}:\hat{H}_i \to H$ is a local isomorphism at a generic point.
\end{df}

Note, that if $f:Y\to X$ is a blow-up with a smooth center $C\subset X,$ then the generic preimage
$H_f$ coincide with the strict transform of $H$ unless $C$ and $H$ has common irreducible
components.

Let $f:Y\to X$ and $g:Z\to Y$ be branched coverings. Let $H\subset X$ be a hypersurface. Then,
clearly, $(H_f)_g=H_{f\circ g}.$ Indeed, the restriction of $f\circ g$ to an irreducible component
of $g^{-1}(f^{-1}(H))$ is a local isomorphism at a generic point if and only if $g$ is a local
isomorphism at a generic point of this component and $f$ is a local isomorphism at a generic point
of the image of this component.

\begin{lem}
$f|_{H_f}:H_f\to H$ is a branched covering.
\end{lem}

\begin{proof}
By Sard's Lemma, the image of those components of $f^{-1}(H)$ which are not components of $H_f$ has
measure $0$ in $H.$ Therefore, $H_f$ is non-empty and its image is dense in $H.$ Then, since $H_f$
is a closed subset in $Y,$ $f|_{H_f}$ is proper. So, we only need to prove that $f|_{H_f}$ is
surjective to $H.$ Indeed, let $x\in H.$ Let $K\subset H$ be a compact neighborhood of $x$ in $H.$
Since $f|_{H_f}$ is proper, $N=f^{-1}(K)\cap H_f$ is compact. So, $f(N)$ is compact and, therefore,
closed. Since $f(H_f)$ is dense in $H,$ there exist a sequence $\{y_m\}\subset H_f,$ such that
$f(y_m)\rightarrow x$ and $\{f(y_m)\}\in K.$ Then $\{y_m\}\subset N$ and $\{f(y_k)\}\subset f(N).$
Therefore, $x\in f(N).$
\end{proof}

Note, that the normalization map is always a branched covering.

\begin{te}\label{degree one covering}
Let $X$ be a pure-dimensional analytic (algebraic) variety. Let $\pi:Y\to X$ be a degree one
branched covering and let $Y$ be normal. Let $p:\widetilde{X}\to X$ be the normalization. Let
$H\subset X$ be a hypersurface. Then there is a natural map $\widetilde{\pi}:H_{\pi}\to H_p$ which
is a degree one branched covering.
\end{te}

\begin{proof}
According to the universal properties of the normalization, the map $\pi:Y\to X$ factors through
the normalization, i.e. there exist a map $\tilde{\pi}:Y\to\widetilde{X}$ such that
$\pi=p\circ\tilde{\pi}.$ Moreover, $(H_p)_{\tilde{\pi}}=H_{\pi}.$ Therefore, it is enough to
consider the case when $X$ is normal. So, $\widetilde{X}=X,$ $H_p=H,$ $\tilde{\pi}=\pi,$ and we
need to prove that $\pi|_{H_{\pi}}:H_{\pi}\to H$ is a degree one branched covering.

Let $sing(X)\subset X$ and $sing(Y)\subset Y$ be the singular loci of $X$ and $Y$ correspondingly.
Since $\pi$ is a proper map and $sing(Y)$ is a closed subvariety, it follows that
$\pi(sing(Y))\subset X$ is a closed subvariety in $X.$ Since $X$ and $Y$ are normal, the
codimensions of $sing(X),$ $sing(Y),$ and, therefore, $\pi(sing(Y))$ are at least $2.$ Let
$X'=X\backslash (sing(X)\cup \pi(sing(Y))).$ Let $H'=H\cap X'.$ Note, that $H'$ is a complement to
a closed subvariety of codimension at least $1$ in $H.$ Let $Y'=\pi^{-1}(X').$ Now $X'$ and $Y'$
are smooth, $\pi':=\pi|_{Y'}:Y'\to X'$ is still a degree one branched covering, and $H'\subset X'$
is a hypersurface. Note, that $H'_{\pi'}=H_{\pi}\cap Y',$ and it is a complement to a subvariety of
codimension at least $1$ in $H_{\pi}$ (although $Y\backslash Y'$ can have codimension $1$).

Let $crit(\pi')\subset Y$ be the critical locus of $\pi.$ Let $\widetilde{X}=X'\backslash
\pi'(crit(\pi')).$ Let $\widetilde{Y}=\pi'^{-1}(\widetilde{X}).$ Then
$\widetilde{\pi}=\pi|_{\widetilde{Y}}:\widetilde{Y}\to\widetilde{X}$ is an isomorphism. Indeed, it
is non-degenerate and, therefore, unbranched covering of degree $1$.

We need the following

\begin{lem}\label{codimension}
Let $M$ and $N$ be analytic manifolds and $f:M\to N$ be a degree one branched covering. Then
$codim(f(crit(f)))\ge 2.$
\end{lem}

\begin{proof}
Suppose that $codim(f(crit(f)))=1.$ Then there exist $p\in crit(f)$ such that
\begin{enumerate}
\item $p$ is a smooth point of the hypersurface $crit(f);$
\item $f^{-1}(f(crit(f)))$ coincide with $crit(f)$ in a neighborhood of $p;$
\item $f(p)$ is a smooth point of the codimension $1$ irreducible component of $f(crit(f));$
\item $f|_{crit(f)}$ is a local isomorphism at $p.$
\end{enumerate}
Therefore, there exist coordinate systems $(x_1,\dots,x_n)$ and $(y_1,\dots,y_n)$ respectively in a
neighborhood $U$ of $p$ in $M$ and in a neighborhood $V$ of $f(p)$ in $N,$ such that
\begin{enumerate}
\item $f(U)\subset V.$
\item $U\cap crit(f)=\{x_n=0\};$
\item $V\cap crit(f)=\{y_n=0\};$
\item $f(x_1,\dots,x_{n-1},0)=(x_1,\dots,x_{n-1},0).$
\end{enumerate}
So, the map $f|_{U}$ is given by
$$
\begin{array}{l}
y_1=x_1+x_n\phi_1(x); \\
\vdots \\
y_{n-1}=x_{n-1}+x_n\phi_{n-1}(x); \\
y_n=x_n^k\phi_n(x),
\end{array}
$$
\noindent where $k\ge 1$ and $\phi_n$ is not divisible by $x_n.$ Note, that $\phi_n(p)\neq 0.$
Indeed, otherwise $\{\phi_n=0\}\subset f^{-1}(f(crit(f)))=\{x_n=0\}$ and $\phi_n$ is not divisible
by $x_n.$

One can get rid of all the $\phi_i$'s simply by changing the coordinates in a neighborhood of $p\in
M.$ Indeed, let
$$
\begin{array}{l}
t_1=x_1+x_n\phi_1(x); \\
\vdots \\
t_{n-1}=x_{n-1}+x_n\phi_{n-1}(x); \\
t_n=x_n\sqrt[k]{\phi_n(x)}.
\end{array}
$$
Easy to check that the Jacobian is not zero, so $(t_1,\dots,t_n)$ are indeed coordinates in a
neighborhood of $p.$ In $t$'s coordinates the map $f$ is given by
$$
\begin{array}{l}
y_1=t_1; \\
\vdots \\
y_{n-1}=t_{n-1}; \\
y_n=t_n^k.
\end{array}
$$
Since a general point in $N$ should have only one primage, $k$ should be equal to $1.$ But then $f$
is non-degenerate at $p.$
\end{proof}

So, the codimension of $\pi'(crit(\pi'))$ is at least $2.$ Therefore,
$\widetilde{H}:=H'\cap\widetilde{X}$ is a complement to a subvariety of degree at least $1$ in
$H'.$ Its' preimage $\widetilde{\pi}^{-1}(\widetilde{H})=\widetilde{H}_{\widetilde{\pi}}$ is also a
complement to a subvariety of codimension at least $1$ in $H'_{\pi'}$. And
$\pi|_{\widetilde{H}}:\widetilde{H}_{\widetilde{\pi}}\to\widetilde{H}$ is an isomorphism. So,
$\pi|_{H_{\pi}}:H_{\pi}\to H$ is a degree one branched covering.
\end{proof}

\subsection{Resolution of Singularities for Flags.}

To avoid difficulties with the resolution of singularities, we need to assume some compactness
condition on the analytic varieties. For simplicity, let us assume that all the analytic varieties
are restrictions of bigger analytic varieties to relatively compact open subsets.

We need the following Theorem which is a direct corollary of the famous Hironaka Theorem (...) on
resolution of singularities:

\begin{te}\label{resolution}
Let $X$ be a variety. Let $Y_1,\dots, Y_K$ be closed subvarieties in $X.$ Then there exist a
branched covering of degree one $\pi:\tilde{X}\to X$ such that:
\begin{enumerate}
\item $\tilde{X}$ is smooth;
\item $\pi|_{\tilde{X}\backslash D}$ is an isomorphism to $reg(X)\backslash (Y_1\cup\dots\cup Y_k),$
where $D=H_1\cup\dots\cup H_N$ is a union of smooth exceptional hypersurfaces $H_i,$ which
simultaneously have only normal crossings. We denote ${\cal D}=\{H_1,\dots,H_N\}$ the set of
exceptional hypersurfaces (let us always assume the exceptional hypersurfaces irreducible).
\item For any $k=1,\dots,K$ $\pi^{-1}(Y_k)$ is a union of hypersurfaces from ${\cal D}$;
\end{enumerate}
\end{te}

In order to improve the resolution, we'll need to do additional blow-ups with centers in
intersections of exceptional hypersurfaces. We'll need some simple properties of this type of
blow-ups:

\begin{lem}\label{combinatorial}
Let $\tilde{X},$ $D=H_1\cup\dots\cup H_N,$ and ${\cal D}=\{H_1,\dots,H_N\}$ be as in Theorem
\ref{resolution}. Let $C=H_{i_1}\cap\dots\cap H_{i_k}$ and $\pi_C:\tilde{X}_C\to\tilde{X}$ be the
blow-up with center in $C.$ Let $H_C=\pi_C^{-1}(C)$ and $\tilde{H}_i\subset \tilde{X}^C$ be the
strict transform of $H_i\in {\cal D}.$ Denote $D_C:=\tilde{H}_1\cup\dots\cup \tilde{H}_N\cup H_C$
and ${\cal D}_C=\{\tilde{H}_1,\dots,\tilde{H}_N,H_C\}.$ Then
\begin{enumerate}
\item $D_C$ again has only normal crossing;
\item $\tilde{H}_{i_1}\cap\dots\cap \tilde{H}_{i_k}=\emptyset;$
\item $\pi_C|_{\tilde{H}_{i_1}\cap\dots\cap \tilde{H}_{i_{k-1}}}$ is an isomorphism to $H_{i_1}\cap\dots\cap H_{i_{k-1}};$
\item if $C\not\subset H_j\in {\cal D}$ (i.e. $j\neq i_1,\dots,i_k$) then $\pi_C^{-1}(H_j)=\tilde{H}_j.$
\end{enumerate}
\end{lem}

\begin{proof}\

\begin{enumerate}
\item is a standard property of blow-ups;
\item is trivial;
\item follows immediately from the standard fact that if $W\supset U\supset V$ are smooth manifolds,
$\pi_V:W_V\to W$ is the blow-up with center in $V$ and $U_V$ is the strict transform of $U,$ then
$\pi_V|_{U_V}: U_V\to U$ is the blow-up with center in $V.$ Indeed, let $W=\tilde{X},$
$U=H_{i_1}\cap\dots\cap H_{i_{k-1}},$ and $V=C.$ Then $U_V=H_{i_1}\cap\dots\cap H_{i_{k-1}}$ and
$V$ is a hypersurface in $U,$ so $\pi_V|_{U_V}$ is an isomorphism to $U;$
\item it follows that $C$ is transversal to $H_j$ because $H_j$ is a hypersurface. In particular,
the normal bundle of $H_j\cap C$ in $H_j$ is canonically isomorphic to the normal bundle to $C$ in
$\tilde{X}$ restricted to $H_j\cap C.$ Therefore, the restriction $\pi_C|_{\pi_C^{-1}(H_j)}$ is the
blow-up of $H_j$ with center in $H_j\cup C.$
\end{enumerate}
\end{proof}

We'll do the additional blow-ups in $n-1$ steps. We call the consequent strict transforms of the
original hypersurfaces $H_1,\dots,H_N$ the $0$-type hypersurfaces and the consequent strict
transforms of the new hypersurfaces appearing after $k$th step --- the $k$-type hypersurfaces.

On the $k$th step we blow-up the intersections of $n-k+1$ $0$-type hypersurfaces. Note, that after
the $(k-1)$th step the $0$-type hypersurfaces cannot meet by more then $n-k+1$ at one point.
Therefore, the centers for the blow-ups on the $k$th step are disjoint and one can blow them up
simultaneously. Also, the $k$-type hypersurfaces are always disjoint for $k>0$ and the $0$-type
hypersurfaces are disjoint after $(n-1)$th step. It is convenient to label the $k$-type
hypersurfaces by $(n-k+1)$-tuples of $0$-type hypersurfaces.

Applying the above procedure, one get the following Lemma:

\begin{lem}\label{better resolution}
In the Theorem \ref{resolution}, one can assume also that ${\cal D}$ satisfy the following
conditions:
\begin{enumerate}
\item Let $H_i,H_j\in {\cal D}$ and $H_i\cap H_j\neq\emptyset.$ Then either $\pi(H_i)\subset\pi(H_j)$
or $\pi(H_i)\supset\pi(H_j);$
\item Let $H_{i_1},\dots,H_{i_k}\in {\cal D},$ $C:=H_{i_1}\cap\dots\cap H_{i_k}\neq\emptyset,$ and
$\pi(H_{i_1})\subset\dots\subset\pi(H_{i_k}).$ Then for any irreducible component $C^0$ of $C$
$\pi(C^0)=\pi(H_1).$
\end{enumerate}
\end{lem}

\begin{proof}
Follows immediately from the Lemma \ref{combinatorial}.
\end{proof}

As an immediate corollary of the Lemma \ref{better resolution}, we get the following result on
resolutions for flags of subvarieties.

\begin{df}
Let $(x_1,\dots,x_n)$ be a coordinate system in $U$ (i.e. $(x_1,\dots,x_n)$ maps $U$ isomorphically
to an open neighborhood of the origin in $\mathbb C^n$). A meromorphic function $f$ is called {\it
almost monomial} in $U$ with respect to $(x_1,\dots,x_n)$ if $f=x_1^{d_1}\dots x_n^{d_n}\phi,$
where $d_1,\dots, d_n$ are integers and $\phi$ is a holomorphic non-zero function on $U.$
\end{df}

\begin{te}\label{resolution for flag}
Let $V_n\supset V_{n-1}\supset\dots\supset V_0$ be a flag of algebraic (closed analytic)
subvarieties with $\dim V_i=i.$ Let $f_1,\dots,f_k$ be meromorphic functions on $V_n.$ Then there
exist a flag of smooth algebraic (closed analytic) subvarieties $\overline{V}_n\supset
\overline{V}_{n-1}\supset\dots\supset \overline{V}_0$ and a map $\pi:\overline{V}_n\to V_n,$ such
that
\begin{enumerate}
\item $\pi$ is a degree one branched covering (in fact, a composition of blow-ups);
\item $\overline{V}_k$ is the generic component of the preimage of $V_k$ with respect to $\pi|_{\overline{V}_{k+1}}$ for $k=n-1,n-2,\dots,0;$
\item for any point $a\in\overline{V}_i$ there exist a system of coordinates (called {\it good coordinates}) $(x_1,\dots,x_n)$ in a neighborhood $U$ of $a$ in $\overline{V}_n$ such that
\begin{enumerate}
\item $\overline{V}_j\cap U=\{b\in U: x_n(a)=\dots=x_{j+1}(a)=0\}$ for $j=i,i+1,\dots,n;$
\item $\pi^*f_1,\dots,\pi^*f_k$ are almost monomial in $U$ with respect to $(x_1,\dots,x_n).$
\end{enumerate}
\end{enumerate}
\end{te}

\begin{proof}
We apply the Theorem \ref{resolution} and Lemma \ref{better resolution} to $V_n$ with subvarieties
$V_{n-1},\dots,V_0$ and all the divisors of $f_1,\dots,f_k.$ Let $\pi:\overline{V}_n\to V_n$ be the
resulting resolution map and ${\cal D}=\{H_1,\dots,H_N\}$ be the exceptional hypersurfaces. Denote
${\cal D}_k=\{H_i\in {\cal D}:\pi(H_i)=V_k\}$ and $D_k=\bigcup\limits_{H_i\in {\cal D}_k}H_i.$ Let
$\overline{V}_n\supset \overline{V}_{n-1}\supset\dots\supset \overline{V}_0$ be the flag of
consequent generic preimages.

\begin{lem}
$\overline{V}_k=D_{n-1}\cap\dots\cap D_k$ and $\overline{V}_k$ is smooth for all $k=0,\dots, n.$
\end{lem}

\begin{proof}
We first prove that $\overline{V}_k=D_{n-1}\cap\dots\cap D_k$ by induction starting from
$\overline{V}_{n-1}=D_{n-1},$ which is trivial by dimension.

Suppose that we already proved $\overline{V}_{k+1}=D_{n-1}\cap\dots\cap D_{k+1}.$ The inclusion
$D_{n-1}\cap\dots\cap D_k\subset \overline{V}_k$ follows from the condition $2$ of the Lemma
\ref{better resolution} and the dimension counting. Indeed, $dim(D_{n-1}\cap\dots\cap
D_k)=n-((n-1)-(k-1))=k=dim(V_k)$ and $\pi(D_{n-1}\cap\dots\cap D_k)=\pi(D_k)$ (unless
$D_{n-1}\cap\dots\cap D_k=\emptyset$ in which case the inclusion is trivial). Therefore, every
irreducible component of $D_{n-1}\cap\dots\cap D_k$ is mapped surjectively to $V_k,$ and, by
dimension, $\pi|_{D_{n-1}\cap\dots\cap D_k}$ is a local isomorphism at a general point.

Suppose now that $X\subset D_{n-1}\cap\dots\cap D_{k+1}$ is an irreducible subvariety of
codimension $1$ such that $\pi(X)=V_k.$ Since $\pi^{-1}(V_k)$ is a union of hypersurfaces from
${\cal D},$ it follows that $X\subset H\in {\cal D},$ such that $\pi(H)\supset V_k.$ Moreover,
there exist $H_{n-1}\in {\cal D}_{n-1},\dots,H_{k+1}\in {\cal D}_{k+1}$ such that $X\in
H_{n-1}\cap\dots\cap H_{k+1}\cap H.$ Consider $\pi(H).$ It is an irreducible subvariety in $V_n$
such that it either contains or is contained in $V_i$ for every $i=n-1,\dots,k+1$ and it contains
$V_k.$ Therefore, it coincide with one of the $V_i$'s for $i=n-1,\dots,k.$ We need to prove that
$\pi(H)=V_k.$ Suppose it is wrong and $\pi(H)=V_i$ for $i>k.$ Then $\pi(H_{n-1}\cap\dots\cap
H_{k+1}\cap H)=V_{k+1}$ which is impossible by dimension.

The only way $D_{n-1}\cap\dots\cap D_k$ can be singular is if two of its irreducible components
intersect each other. That means that there are two different sets of hypersurfaces
$H_{n-1},\dots,H_k$ and $H'_{n-1},\dots,H'_k$ with $H_i,H'_i\in {\cal D}_i,$ such that
$C:=H_{n-1}\cap\dots\cap H_k\cap H'_{n-1}\cap\dots\cap H'_k\neq\emptyset.$ Since there is at least
$n-k+1$ different hypersurfaces, $dim(C)<k.$ However, by the condition $2$ of the Lemma \ref{better
resolution}, $\pi(C)=V_k$ which is impossible by dimension.
\end{proof}

The Theorem \ref{resolution for flag} is proved.

\end{proof}

\begin{df}
The flag $\overline{V}_n\supset \overline{V}_{n-1}\supset\dots\supset \overline{V}_0$ together with
the map $\pi:\overline{V}_n\to V_n$ is called the {\it resolution of singularities of the flag
$V_n\supset V_{n-1}\supset\dots\supset V_0,$ respecting the functions $f_1,\dots,f_k.$}
\end{df}
%We'll need one more technical lemma:

%\begin{lem}
%Let $f:X\to Y$ be a degree $1$ branched covering. Let $H\subset X$ be a hypersurface and let
%$f|_{H}:H\to f(H)$ be also a degree $1$ branched covering. Let also $H\cap reg(X)\neq\emptyset$ and
%$f(H)\cap reg(Y)\neq\emptyset.$ Suppose that a meromorphic function $g$ has degree $k$ at a general
%point of $f(H).$ Than $g\circ f$ has degree $k$ at a general point of $H.$
%\end{lem}

%\begin{proof}
%Obviously follows from the Lemma \ref{codimension}. Indeed, general point of $f(H)$ doesn't belong
%to $f(crit(f))$
%\end{proof}

\subsection{Parshin's point, local parameters, and residue.}

Let $V_n$ be an algebraic (analytic) variety of dimension $n.$ Let $V_n\supset\dots\supset V_0$ be
a flag of subvarieties of dimension $dim V_k=k$ in $V_n.$

Consider the following diagram:

\begin{equation}\label{NormDiag}
\begin{CD}
V_n             @.\supset @. V_{n-1}                     @.\supset @.  \dots @.\supset @. V_1                            @.\supset @. V_0\\
@AAp_nA              @.      @AAp_nA                          @.        @.       @.             @.                           @.       @. \\
\widetilde{V}_n @.\supset @. W_{n-1}                     @.        @.       @.         @.                                @.        @.    \\
@.                   @.      @AAp_{n-1}A                      @.        @.       @.             @.                           @.       @. \\
                @.        @. \widetilde{W}_{n-1}         @.\supset @. \dots @.        @.                                @.        @.   \\
@.                   @.      @.                               @.        @.       @.             @.                           @.       @. \\
                @.        @.                             @.        @. \dots @.\supset @. W_1                            @.        @.   \\
@.                   @.      @.                               @.        @.       @.       @AAp_{1}A                          @.       @. \\
                @.        @.                             @.        @.       @.        @. \widetilde{W}_1                @.\supset @.  W_0
\end{CD}
\end{equation}

\noindent where
\begin{enumerate}
\item $p_n:\widetilde{V}_n\to V_n$ is the normalization;
\item $W_{n-1}\subset\widetilde{V}_n$ is the general preimage of $V_{n-1}$ under $p_n;$
\item for every $k=1,2\dots,n-1$
\begin{enumerate}
\item $p_k:\widetilde{W_k}\to W_k$ is the normalization;
\item $W_{k-1}\subset\widetilde{W}_k$ is the general preimage of $V_{k-1}$ under $p_n\circ\dots\circ p_{k}.$
\end{enumerate}
\end{enumerate}

\begin{df}
We call the diagram \ref{NormDiag} the {\it normalization diagram} of the flag
$V_n\supset\dots\supset V_0.$
\end{df}

\begin{df}
The flag $V_n\supset\dots\supset V_0$ of irreducible subvarieties together with a choice of a point
$a_{\alpha}\in W_0$ is called a {\it Parshin's point}.
\end{df}

\noindent{\bf Remark.} Consider $\{V_n\supset\dots\supset V_0,a_{\alpha}\in W_0\}$ where $V_i$'s
are not necessary irreducible. The choice of $a_{\alpha}$ provides the irreducible components of
$V_i$'s. Therefore, $\{V_n\supset\dots\supset V_0,a_{\alpha}\in W_0\}$ defines a Parshin's point
$\{V'_n\supset\dots\supset V'_0,a_{\alpha}\in W'_0\},$ where $V'_i$ are the corresponding
irreducible components and $W'_0\subset W_0.$ (Note that the diagram (\ref{NormDiag}) for the
irreducible flag $V'_n\supset\dots\supset V'_0$ consists of irreducible components of the elements
of the diagram (\ref{NormDiag}) for the original flag $V_n\supset\dots\supset V_0$.) However, if we
wouldn't require irreducibility of $V_i$'s, $\{V_n\supset\dots\supset V_0,a_{\alpha}\}$ and
$\{V'_n\supset\dots\supset V'_0,a_{\alpha}\}$ would formally be different Parshin's points, which
is wrong. Sometimes we'll say Parshin's point $\{V_n\supset\dots\supset V_0,a_{\alpha}\}$ without
assuming that $V_i$'s are irreducible, by using the convention $\{V_n\supset\dots\supset
V_0,a_{\alpha}\}=\{V'_n\supset\dots\supset V'_0,a_{\alpha}\}.$

$W_{i-1}\subset \widetilde{W}_i$ is a hypersurface in a normal variety. It follows that there exist
a (meromorphic) function $u_i$ on $\widetilde{W}_i$ which has zero of order $1$ at a generic point
of $W_{i-1}.$ Since meromorphic functions are the same on $W_i$ and $\widetilde{W}_i,$ one can
consider $u_i$ as a function on $W_i.$ Then one can continue $u_i$ to $\widetilde{W}_{i+1}$ and so
on. For simplicity, we denote all these functions by $u_i.$ Now $u_i$ is defined on $V_n,$ and
$W_j$ and $\widetilde{W}_j$ for $j\ge i.$

\begin{df}
Functions $(u_1,\dots,u_n)$ are called {\it local parameters}.
\end{df}

Let $\omega$ be a meromorphic $n$-form on $V_n.$ Since the differentials $du_1,\dots,du_n$ are
linearly independent at a generic point of $V_n$, one can write

$$
\omega=fdu_1\wedge\dots\wedge du_n,
$$

\noindent where $f$ is a meromorphic function on $V_n.$

Furthermore, $f$ can be expand into a power series in $u_n$ with coefficients in meromorphic
functions on $W_{n-1}.$ Let $f_{-1}$ be the coefficient at $u_n^{-1}$ in this expansion. Then
$\omega_{n-1}=f_{-1}du_1\wedge\dots\wedge du_{n-1}$ is a meromorphic form on $W_{n-1}.$ Continuing
in the same way, we get a function $\omega_0$ on the finite set $W_0.$

\begin{df}
The residue of $\omega$ at the Parshin's point $P=\{V_n\supset\dots\supset V_0,a_{\alpha}\in W_0\}$
is $res_P(\omega):=\omega_0(a_{\alpha}).$
\end{df}

\subsection{Toric Parameters and Toric Neighborhoods.}

Consider the following diagram:
$$
\begin{CD}
\overline{V}_n  @.\supset @. \overline{V}_{n-1}          @.\supset @.  \dots @.\supset @. \overline{V}_1                 @.\supset @. \overline{V}_0\\
@VV\pi V              @.      @VV\pi V                          @.        @.       @.       @VV\pi V                           @.       @VV\pi V \\
V_n             @.\supset @. V_{n-1}                     @.\supset @.  \dots @.\supset @. V_1                            @.\supset @. V_0=\{a\}\\
@AAp_nA              @.      @AAp_nA                          @.        @.       @.             @.                           @.       @. \\
\widetilde{V}_n @.\supset @. W_{n-1}                     @.        @.       @.         @.                                @.        @.    \\
@.                   @.      @AAp_{n-1}A                      @.        @.       @.             @.                           @.       @. \\
                @.        @. \widetilde{W}_{n-1}         @.\supset @. \dots @.        @.                                @.        @.   \\
@.                   @.      @.                               @.        @.       @.             @.                           @.       @. \\
                @.        @.                             @.        @. \dots @.\supset @. W_1                            @.        @.   \\
@.                   @.      @.                               @.        @.       @.       @AAp_{1}A                          @.       @. \\
                @.        @.                             @.        @.       @.        @. \widetilde{W}_1                @.\supset @.  W_0
\end{CD}
$$
\noindent where the first line is the resolution of singularities of the flag
$V_n\supset\dots\supset V_0$ and the rest is the normalization diagram.

The following Lemma follows immediately from the Theorem \ref{degree one covering}:

\begin{lem}\label{birational uniqueness of resolution}
For every $k=0,\dots,n-1$ there exist a natural degree one branched covering
$\varphi_k:\overline{V}_k\to\widetilde{W}_k,$ such that all the diagram commutes. In particular,
$\varphi_0:\overline{V}_0\to\widetilde{W}_0=W_0$ is a bijection of finite sets.
\end{lem}

Let $(u_1,\dots,u_n)$ be local parameters. Let $b_{\alpha}\in \overline{V}_0.$ According to the
Theorem \ref{resolution for flag} there exist good coordinates $(x_1,\dots,x_n)$ near $b_{\alpha},$
such that $(v_1,\dots,v_n):=(u_1\circ\pi,\dots,u_n\circ\pi)$ are almost monomial near $b_{\alpha}.$
Let $v_k=x_1^{d_{k1}}\dots x_n^{d_{kn}}\phi_k,$ for $k=1,\dots,n,$ where $\phi_1,\dots,\phi_n$ are
non-zero holomorphic functions in a neighborhood of $b_{\alpha}.$

\begin{lem}\label{local param. -> gen. local param.}
$d_{ij}=0$ for $i<j$ and $d_{ii}=1$ for $i=1,\dots,n.$
\end{lem}

\begin{proof}
According to the definition of the local parameters, $u_m$ can be restricted consequently from
$\widetilde{V}_n$ to $W_{n-1},$ from $\widetilde{W}_{n-1}$ to $W_{n-2}$ and so on down to $W_m$ and
all the restrictions are not identical zeros. Finally, on $\widetilde{W}_m,$ $u_m$ has zero of
order $1$ at a general point of $W_{m-1}.$

$\varphi_k:\overline{V}_k\to\widetilde{W}_k$ is a local isomorphism at a generic point of
$\overline{V}_{k-1}$ (follows from Lemma \ref{codimension} applied to
$\varphi_k|_{\varphi_k^{-1}(reg(\widetilde{W}_k))}$). Therefore, $v_k$ restricts to a meromorphic
function on $\overline{V_k}$ with zero of order $1$ at a generic point of $\overline{V}_{k-1}.$
\end{proof}

\begin{df}
The matrix $D:=\{d_{ij}\}$ is called the {\it valuation matrix} of the almost monomial functions
$(v_1,\dots,v_n)$ with respect to $(x_1,\dots,x_n).$
\end{df}

\begin{df}
A set $(y_1,\dots,y_n)$ of almost monomial functions in a neighborhood of
$b_{\alpha}\in\overline{V}_n$ with respect to the coordinates $(x_1,\dots,x_n),$ such that the
corresponding valuation matrix $D$ is lower-diagonal with units on diagonal is called {\it
generalized local parameters.}
\end{df}

Lemma \ref{local param. -> gen. local param.} basically says that if $(u_1,\dots, u_n)$ are local
parameters and the resolution $\pi:\overline{V}_n\to V_n$ is such that
$(v_1,\dots,v_n):=(u_1\circ\pi,\dots,u_n\circ\pi)$ are almost monomial near $b_{\alpha}$ then
$(v_1,\dots,v_n)$ are generalized local parameters. However, later we'll have to use generalized
local parameters of more general type.

We have the following simple lemma about the almost monomial functions:

\begin{lem}
Let $(y_1,\dots,y_n)$ be almost monomial functions with respect to the good coordinates
$(x_1,\dots,x_n)$ in a neighborhood of $b_{\alpha}.$ Let $D=\{d_{ij}\}$ be the valuation matrix.
Let $C=\{c_{ij}\}=D^{-1},$ and let
$$
\begin{array}{l}
z_1=y_1^{c_{11}}\dots y_1^{c_{1n}};\\
z_2=y_1^{c_{21}}\dots y_1^{c_{2n}};\\
\vdots \\
z_n=y_1^{c_{n1}}\dots y_n^{c_{nn}}.
\end{array}
$$
Then $(z_1,\dots,z_n)$ are good coordinates in a neighborhood of $b_{\alpha}$ as well. (If $|det
D|\neq 1$ then $C$ has rational entries. However, $z_1,\dots,z_n$ don't have any branching near
$b_{\alpha},$ so one should just choose one branch for each $z_i$.)
\end{lem}

\begin{proof}
Indeed, the valuation matrix of $(z_1,\dots,z_n)$ is the identity matrix, so
$$
\begin{array}{l}
z_1=x_1h_1;\\
z_2=x_2h_2;\\
\vdots \\
z_n=x_nh_n,
\end{array}
$$
\noindent where $h_1,\dots,h_n$ are holomorphic non-zero functions in a neighborhood of
$b_{\alpha}.$ Then
$$
\frac{\partial z_i}{\partial x_j}|_{b_{\alpha}}=\delta_{ij}h_1(b_{\alpha}),
$$
\noindent so,
$$
|J(b_{\alpha})|=|\{\frac{\partial z_i}{\partial x_j}\}|=h_1(b_{\alpha})\cdot\dots\cdot
h_n(b_{\alpha})\neq 0.
$$
\noindent Therefore, by the Inverse Function Theory, $(z_1,\dots,z_n)$ are coordinates in a
neighborhood of $b_{\alpha}.$ Also, $\{z_i=0\}=\{x_i=0\}$ in a neighborhood of $b_{\alpha}$ for any
$i=1,\dots,n.$ So, $(z_1,\dots,z_n)$ are good coordinates.
\end{proof}

Let $(y_1,\dots,y_n)$ be generalized local parameters and let $D=\{d_{ij}\}$ be the corresponding
valuation matrix. Let $\hat{L}^n$ be the lattice of monomials in $y_1,\dots, y_n$ endowed with the
lexicographic order with respect to the basis $(y_1,\dots,y_n)$ (first with respect to $y_n,$ then
$y_{n-1}$ and so on). Consider the flag $\hat{L}^n\supset \hat{L}^{n-1}\supset\dots\supset
\hat{L}^0$ of isolated subgroups in $\hat{L}^n.$ Let ${\cal \hat{C}}$ be the corresponding system
of cones, $\hat{L}$ be the corresponding semigroup, and ${\cal \hat{T}}$ be the corresponding
system of toric varieties. Note, that all the toric varieties in ${\cal \hat{T}}$ are normal and,
moreover, isomorphic to $\mathbb C^n.$

Note that the valuation matrix $D$ and its inverse $C$ are lower-triangular integer matrix with
units on the diagonal in this case. Therefore,
$$
\begin{array}{l}
z_1=y_1^{c_{11}}\dots y_1^{c_{1n}};\\
z_2=y_1^{c_{21}}\dots y_1^{c_{2n}};\\
\vdots \\
z_n=y_1^{c_{n1}}\dots y_n^{c_{nn}},
\end{array}
$$
\noindent are standard coordinates on a variety $\hat{T}_z\in {\cal \hat{T}}.$

Therefore, $(z_1,\dots,z_n)$ provides an isomorphism $\psi^{-1}:W\to \hat{U}$ between a
neighborhood $W$ of $b_{\alpha}\in\overline{V}_n$ and a neighborhood $\hat{U}$ of the origin in
$\hat{T}_z.$ Moreover, $\psi(\hat{U}^k)\subset\overline{V}_k,$ where $\hat{U}^k=\hat{T}^k\cap
\hat{U},$ and $\psi(\hat{U}^n)=W\backslash \{z_1\cdot\dots\cdot z_n=0\}.$

Let $\hat{\phi}:=\pi\circ\psi:\hat{U}\to V_n.$ We proved the following

\begin{te}\label{non-toric neighborhood}
The map $\hat{\phi}$ has the following properties:
\begin{enumerate}
\item $\hat{\phi}(\hat{U}^k)\subset V_k$ for $k=0,\dots,n$;
\item $\hat{\phi}|_{\hat{U}^n}$ is an isomorphism to the image;
\item for any hypersurface $H\subset V^n$ one can choose the resolution $\pi:\overline{V}_n\to V_n$
in such a way, that $H\cap \hat{\phi}(\hat{U}^n)=\emptyset.$
\end{enumerate}
\end{te}

%Property $4$ follows from the condition on the exceptional divisors in the Theorem \ref{resolution
%for flag}. The restriction $\pi|_{\overline{V}_k}$ can only be degenerate in the points of
%intersection of $\overline{V}_k$ with the exceptional divisors, not containing $\overline{V}_k.$

\setcounter{cor}{0}

\begin{cor}
Let $f$ be a meromorphic function in a neighborhood of $a\in V^n.$ One can choose the variety
$\hat{T}\in {\cal \hat{T}},$ neighborhood of the origin $\hat{U}\subset \hat{T},$ and the map
$\hat{\phi}: \hat{U}\to V^n$ in such a way that there exist a power series $f_v\in F(\hat{L})$
converging to $f\circ\hat{\phi}$ in $\hat{U}^n.$ In other words, $f$ expands into a power series in
$(v_1,\dots,v_n)$ normally converging to $f$ in $\hat{\phi}(\hat{U}^n),$ and the Newton's
polyhedron of this power series belong to a cone from ${\cal \hat{C}},$ shifted by an integer
vector.
\end{cor}

\begin{cor}
Parameters $(v_1,\dots,v_n)$ induce the homomorphism $v^*:F(V^n)\to F(\hat{L})$ from the field of
meromorphic functions on $V^n$ to the field $F(\hat{L})$ of Laurent power series of $\hat{L}.$ In
particular, the field of meromorphic functions $F(V^n)$ is endowed with the valuation
$\nu_v=\nu_{F(\hat{L})}\circ v^*.$
\end{cor}

\begin{cor}
Let $\omega$ be a meromorphic $n$-form on $V_n.$ Let $\omega=fdv_1\wedge\dots\wedge v_n.$ Then
$res_{(V^n\supset\dots\supset
V_0,a_{\alpha}=\phi_0(b_{\alpha}))}(\omega)=res(v^*(f)dv_1\wedge\dots\wedge dv_n).$
\end{cor}

\begin{cor}
Let $(v'_1,\dots,v'_n)$ be another set of local parameters. Then the set of Laurent series
$(v^*(v'_1),\dots,v^*(v'_n))$ define the change of variables from ${\cal \hat{T}}$ to ${\cal
\hat{T}'}.$ Moreover, $v'^*=\psi\circ v^*,$ where $\psi:F(\hat{L})\to F(\hat{L}')$ is the
isomorphism given by the change of variables.
\end{cor}

\begin{cor}
Parshin's residue doesn't depend on the choice of local parameters.
\end{cor}

Note, that $\hat{\phi}|_{\hat{U}^k}$ for $k<n$ is not an isomorphism to the image. It is rather a
covering, at least locally (one should be careful defining what this locality actually means).
Indeed, $\hat{\phi}=\pi\circ \psi,$ where $\psi|_{\hat{U}^k}$ is an isomorphism to the image for
all $k=0,\dots,n,$ while $\pi|_{\overline{V}_k}$ can be a branched covering of a higher degree with
some branching near our flag. We want to improve it in such a way that $\phi|_{U^k}$ is an
isomorphism to the image for all $k=0,\dots,n.$ In particular, it will help us to understand the
local geometry of $V^n$ near the flag $V^n\supset\dots\supset V^0.$ In order to do so, one needs to
consider special generalized local parameters, which we call {\it toric parameters} and more
general systems of (non-normal) toric varieties, associated to flags of lattices, different from
the flag of isolated subgroups.

Let $\rho_k:\widetilde{V}_k\to V_k$ be the normalization maps for $k=0,\dots,n$ ($\rho_n=p_n$). Let
$\hat{V}_{k-1}\subset\widetilde{V}_k$ be general primages of hypersurfaces $V_{k-1}\subset V_k$
under $\rho_k.$ For every $k=1,\dots,n,$ let $t_k$ be a meromorphic function on $\widetilde{V}_k$
which has zero of order $1$ at a general point of $\hat{V}_{k-1}.$ One can think of $t_k$ as a
meromorphic function on $V_k$ as well and, moreover, for every $k$ let us continue $t_k$ to a
meromorphic function on the whole $V_n.$ For simplicity, we'll denote this continuation by $t_k$ as
well.

By applying the Theorem \ref{resolution for flag} several times one can get the following diagram:

$$
\begin{CD}
V_n^n   @.\supset @. V_{n-1}^n @.\supset @. \dots @.\supset @. V_1^n @.\supset @. V_0^n \\
@VV\pi_nV   @.      @VV\pi_nV      @.        @.       @.      @VV\pi_nV  @.      @VV\pi_nV \\
\vdots  @.        @. \vdots    @.        @.       @.        @. \vdots @.       @. \vdots \\
@VV\pi_3V   @.      @VV\pi_3V      @.        @.       @.      @VV\pi_3V  @.      @VV\pi_3V \\
V_n^2   @.\supset @. V_{n-1}^2 @.\supset @. \dots @.\supset @. V_1^2 @.\supset @. V_0^2 \\
@VV\pi_2V   @.      @VV\pi_2V      @.        @.       @.      @VV\pi_2V  @.      @VV\pi_2V \\
V_n^1   @.\supset @. V_{n-1}^1 @.\supset @. \dots @.\supset @. V_1^1 @.\supset @. V_0^1 \\
@VV\pi_1V   @.      @VV\pi_1V      @.        @.       @.      @VV\pi_1V  @.      @VV\pi_1V \\
V_n     @.\supset @. V_{n-1}   @.\supset @. \dots @.\supset @. V_1   @.\supset @. V_0
\end{CD}
$$

\noindent where $(\pi_1\circ\pi_2\circ\dots\circ\pi_k)|_{V_k^k}:V_k^k\to V_k$ is a resolution of
singularities of the flag $V_k\supset\dots\supset V_0$ respecting the functions $t_1,\dots,t_k.$

Let $b_{\alpha}\in V_0^n$ as before. Let $(x_1,\dots,x_k)$ be good coordinates in a neighborhood of
the $b_{\alpha}^k:=\pi_{k+1}\circ\dots\circ\pi_{n}(b_{\alpha}).$ Let
$\tilde{t}_i:=(\pi_1\circ\dots\circ\pi_k)^*(t_i)$ for $i=1,\dots,k.$ Let $D^k=\{d_{ij}^k\}$ be the
valuation matrix of $(\tilde{t}_1,\dots,\tilde{t}_k)$ with respect to $(x_1,\dots,x_k)$ and let
$C^k=\{c_{ij}^k\}=(D^k)^{-1}$ be the inverse matrix.

\begin{lem}
$d_{ij}^k=0$ for $i<j,$ $d_{kk}^k=1,$ and $d_{ii}^k\neq 0$ for $i=1,\dots,k-1.$
\end{lem}

\begin{proof}
Similar to Lemma \ref{local param. -> gen. local param.}.
\end{proof}

Let
$$
\begin{array}{l}
t_1^k=t_1^{c_{11}^k}\dots t_1^{c_{1k}^k};\\
t_2^k=t_1^{c_{21}^k}\dots t_1^{c_{2k}^k};\\
\vdots \\
t_k^k=t_1^{c_{k1}^k}\dots t_n^{c_{kk}^k},
\end{array}
$$
\noindent and let $L^k$ be the lattice of monomials in $(t_1^k,\dots,t_k^k).$ Here $L_k$ is a
subgroup in the multiplicative group of monomials in $(t_1,\dots,t_n)$ with rational powers. Note,
that although $t_i^k$ are multivalued functions globally, they don't any branching near
$b_{\alpha}\in V_n^n.$ We choose the branches simultaneously for all $k=1,\dots,n$. By abuse of
notations we denote these branches by $t_i^k$ as well.

\begin{lem}
$L^n\supset\dots\supset L^1\supset L^0:=\{0\}$ is a flag of lattices, where $L^n$ is endowed with
the lexicographic order with respect to the basis $(t_1^n,\dots,t_n^n)$ (first with respect to
$t_n^n,$ then $t_{n-1}^n,$ etc.). Moreover, the cone generated by $(t_1^{k-1},\dots,t_{k-1}^{k-1})$
is a subset of the cone generated by $(t_1^k,\dots,t_k^k).$
\end{lem}

\begin{proof}
It is enough to prove that $(t_1^{k-1},\dots,t_{k-1}^{k-1})$ are monomials with positive integer
powers in $(t_1^k,\dots,t_{k-1}^k).$ Indeed, we have the map $\pi_k|_{V_{k-1}^k}: V_{k-1}^k\to
V_{k-1}^{k-1},$ which is continuous. $(t_1^{k-1},\dots,t_{k-1}^{k-1})$ are coordinates in a
neighborhood of $b_{\alpha}^{k-1}=\pi_k(b_{\alpha}^k)$ in $V_{k-1}^{k-1}$ and
$(t_1^k,\dots,t_{k-1}^k)$ are coordinates in a neighborhood of $b_{\alpha}^k$ in $V_{k-1}^k.$
Therefore, the powers have to be positive and integer.
\end{proof}

Let ${\cal C},$ $L,$ and ${\cal T}$ be, respectively, system of cones, semigroup, and system of
toric varieties, associated to the flag of lattices $L^n\supset\dots\supset L^1\supset L^0.$

Note, that $t_1^n,\dots,t_n^n$ are generalized local parameters (here we consider
$t_1^n,\dots,t_n^n$ as functions in a neighborhood of $b_{\alpha}\in V_n^n$). So, one can apply the
Theorem \ref{non-toric neighborhood} to them.

Note, that lattice $\hat{L}^n=L^n.$ So, the flag of isolated subgroups $\hat{L}^n\supset
\hat{L}^{n-1}\supset\dots\supset \hat{L}^0$ is the normalization of the flag of lattices
$L^n\supset\dots\supset L^1\supset L^0.$ Let $\hat{\phi}:\hat{U}\to V_n$ be the map from the
Theorem \ref{non-toric neighborhood}. Here $\hat{U}\subset \hat{T}_{t^n}\in\hat{\cal T}$ is a
neighborhood of the origin. Let $T_{t^n}\in {\cal T}$ be the variety which normalization is
$\hat{T}_{t^n}$ and let $\nu:\hat{U}\to U$ be the restriction of the normalization map to $\hat{U}$
($U=\nu(\hat{U})$). Let $U^k=\nu(\hat{U}^k)$ (i.e. the intersection of $U$ with the corresponding
toric orbit in $T_{t^n}$).

\begin{te}
The map $\hat{\phi}:\hat{U}\to V_n$ factors through the normalization map $\nu:\hat{U}\to U,$ i.e.
$\hat{\phi}=\phi\circ\nu$ where $\phi:U\to V_n.$ Moreover, $\phi|_{U^k}$ is an isomorphism to the
image inside $V_k$ for all $k=0,1,\dots,n.$
\end{te}

\begin{proof}
According to the Lemma \ref{almost normal} it is enough to show that $\hat{\phi}$ factors through
the normalization on the level of sets and that $\phi$ is continuous (then it is regular). More
over, similarly to the Theorem \ref{change of variables} the continuity follows immediately. The
only thing we need to check is that $\phi$ is well defined on the level of sets, i.e. that for any
$x\in U$ and $y_1,y_2\in\nu^{-1}(x)$ we have $\hat{\phi}(y_1)=\hat{\phi}(y_2).$ Indeed, let $k$ be
the smallest number such that $x\in \overline{U^k}.$ The functions $t_k^1,\dots,t_k^k$ provides
coordinates on $V_k^k$ near $b_{\alpha}^k.$ On the other hand, they belong to $L_k$ and, therefore,
they are regular on $U\backslash\overline{U^{k-1}},$ in particular, at $x.$

Note, that $\hat{\phi}=\pi_1\circ\dots\circ\pi_n\circ\psi.$ Denote
$z_i=\pi_{k+1}\circ\dots\circ\pi_n\circ\psi(y_i)\in V_k^k.$ It follows from the above, that
$t_k^1,\dots,t_k^k$ don't distinguish $z_1$ and $z_2.$ Therefore, $z_1=z_2.$

Consider now the restriction $\phi|_{U^k}.$ It follows from the above that
$\phi|_{U^k}=\pi_1\circ\dots\circ\pi_k\circ\psi_k,$ where $\psi_k$ is a regular map from $U^k$ to
the complement of the coordinate cross in a neighborhood of $b_{\alpha}^k$ in $V_k^k.$ Moreover,
$\psi_k$ is an isomorphism to the image (since $dt_k^1\wedge\dots\wedge dt_k^k$ is a non-degenerate
top form both in the image and the preimage). $\pi_1\circ\dots\circ\pi_k$ is an isomorphism to the
image on the complement to the exceptional divisor. Therefore, $\phi|_{U^k}$ is an isomorphism to
the image as well.
\end{proof}

\end{document}